\documentclass[11pt]{amsart}
\newcommand{\ben}{\begin{enumerate}}
\newcommand{\een}{\end{enumerate}}
\newcommand{\eq}[2][label]{\begin{equation}\label{#1}#2\end{equation}}
\newcommand{\av}[2]{\langle #1\rangle_{_{\scriptstyle #2}}}
\newcommand{\ave}[1]{\langle #1\rangle}

\newcommand{\sav}[2]{\langle #1\rangle_{_{\scriptscriptstyle #2}}}

\DeclareMathOperator\sign{sgn}

\newcommand{\ve}{\varepsilon}
\usepackage{amsmath}
\usepackage{amsfonts}
\usepackage{amssymb}
\usepackage{amsthm}
\usepackage{graphicx,color,hyperref}

\usepackage{enumitem}

\newcommand{\cre}[1]{{\color{black}#1}}

\newcommand{\bel}[1]{\boldsymbol{#1}}

\newcommand{\ma}{Monge--Amp\`{e}re }

\newcommand{\BMO}{{\rm BMO}}

\newcommand{\overbar}[1]{\mkern 1.5mu\overline{\mkern-1.5mu#1\mkern-1.5mu}\mkern 1.5mu}

\newcommand{\acr}{\overbar{v}}
\newcommand{\bcr}{\overbar{w}}
\newcommand{\xicr}{1-\frac1{3^{p+2}\Gamma(p)}}

\newtheorem{theorem}{Theorem}[section]

\newtheorem{lemma}[theorem]{Lemma}

\newtheorem*{theorem*}{Theorem}{\bf}{\it}
\newtheorem*{proposition*}{Proposition}{\bf}{\it}
\newtheorem*{observation*}{Observation}{\bf}{\it}
\newtheorem*{lemma*}{Lemma}{\bf}{\it}

\theoremstyle{definition}

\theoremstyle{remark}
\newtheorem{remark}[theorem]{Remark}

\numberwithin{equation}{section}

\allowdisplaybreaks
\begin{document}

\title{The John--Nirenberg constant of $\BMO^p,$ $p> 2$}

\author{Leonid Slavin}
\address{University of Cincinnati}
\email{leonid.slavin@uc.edu}

\author{Vasily Vasyunin}
\address{St.~Petersburg Department of the V.~A.~Steklov
Mathematical Institute, RAS, \and St.~Petersburg State University}
\email{vasyunin@pdmi.ras.ru}

\thanks{Both authors were supported by the Russian Science Foundation grant 14-41-00010}



\begin{abstract}
This paper is a continuation of earlier work by the first author who determined the John--Nirenberg constant of $\BMO^p\big((0,1)\big)$ for the range $1\le p\le 2.$ Here, we compute that constant for $p>2.$ As before, the main results rely on Bellman functions for the $\cre{L^p}$ norms of logarithms of $A_\infty$ weights, but for $p>2$ these functions turn out to have a significantly more complicated structure than for $1\le p\le 2.$ 
\end{abstract}

\maketitle


\section{Preliminaries and main results}
\label{main_results}
For a finite interval $J$ and a function $\varphi\in L^1(J),$ let $\av{\varphi}J$ denote the average of $\varphi$ over $J$ with respect to the Lebesgue measure,
$
\av{\varphi}J=\frac1{|J|}\int_J\varphi.
$
\cre{Take an interval $Q$ and $p>0,$} and let $\BMO(Q)$ be the (factor-)space
\eq[1]{
\BMO(Q)=\{\varphi\in L^1\cre{(Q)}\colon\quad\|\varphi\|_{\BMO^p(Q)}:=\sup_{\text{interval}~J\subset Q}\av{|\varphi-\av{\varphi}J|^p}J^{1/p}<\infty\}.
}
It is a classical fact that all $p$-based (quasi-)norms are equivalent, which justifies omitting the index $p$ in the left-hand side.

A weight is an almost everywhere positive function. \cre{We say that a weight $w$ belongs to $A_\infty(Q),$ $w\in A_\infty(Q),$ if both $w$ and $\log w$ are integrable on $Q$ and the following condition holds:
$$
[w]^{}_{A_\infty(Q)}:=\sup_{\text{interval~}J\subset Q}\av{w}Je^{-\av{\log w}J}<\infty.
$$
The quantity $[w]_{A_\infty(Q)}$ is called the $A_\infty(Q)$-characteristic of $w.$
When $Q$ is fixed or unimportant, we write simply $\BMO$ for $\BMO(Q)$ and $A_\infty$ for $A_\infty(Q).$}

BMO functions are locally exponential integrable. We can state this property in the form of the so-called integral John--Nirenberg inequality, which is a variant of the classical weak-type inequality from~\cite{jn}.
\begin{theorem*}[John--Nirenberg]
\cre{Take $p>0.$ There exists a number $\ve_0(p)>0$ such that if $\ve\in[0,\ve_0(p)),$ $Q$ is an interval, and $\varphi\in\BMO(Q)$ with $\|\varphi\|_{\scriptscriptstyle\BMO^p(Q)}\le\ve,$
then there is a number $C(\ve,p)>0$ such that for any interval $J\subset Q,$
\eq[jn2]{
\av{e^\varphi}J \le C(\ve,p)e^{\av{\varphi}{\scriptscriptstyle J}}.
}
}
\end{theorem*}
\noindent We will always use $\ve_0(p)$ to denote the best -- largest possible -- constant in this theorem and call it the John--Nirenberg constant of $\BMO^p$ \cre{(on an interval). Likewise, $C(\ve,p)$ will denote the smallest possible constant in~\eqref{jn2}.}

Observe that~\eqref{jn2} means that if $\varphi\in\BMO,$ then $e^{\ve\varphi}\in A_\infty$ for all sufficiently small $\ve>0.$ For $\varphi\in\BMO,$ let
\eq[ephi]{
\cre{\ve_\varphi=\sup\{\ve:~e^{\ve \varphi}\in A_\infty\}.}
}
In fact, it can be shown that
$$
\cre{\ve_0(p)=\inf\{ \ve_\varphi:~\|\varphi\|_{\BMO^p}\!=\!1\}=\sup\{\ve:~\forall\varphi,~\|\varphi\|_{\BMO^p}\!=\!1~\Longrightarrow ~e^{\ve\varphi}\in A_\infty\}}.
$$

In this paper, our goal is to compute $\ve_0(p)$ for the case $p>2.$ Here are some previous results in that direction: Korenovskii~\cite{korenovsky} and Lerner~\cite{lerner} computed the analogs for the weak-type John--Nirenberg inequality of $\ve_0(1)$ and $C(\ve,1),$ respectively; in~\cite{sv}\cre{,} we determined $\ve_0(2)$ and $C(\ve,2);$ in~\cite{vv}, the second author and \cre{A.}~Volberg found all constants in the weak-type inequality for $p=2;$ and, finally, in~\cite{p<2}, the first author determined $\ve_0(p)$ for $p\in[1,2]$ and $C(\ve,p)$ for $p\in(1,2]$ and large enough $\ve.$ This last paper built the framework that we follow here, and we refer the reader  to it for an in-depth discussion of the tools involved and the differences between the cases $p=2$ and $p\ne 2.$ 

Let us state the relevant theorem from~\cite{p<2}.
\begin{theorem}[\cite{p<2}]
\label{t0}
For $p\in[1,2]$,
$$
\ve_0(p)=\left[\frac pe\Big(\Gamma(p)-\int_0^1t^{p-1}e^t\,dt\Big)+1\right]^{1/p}.
$$
Furthermore,
if $1<p\le 2,$ then for all $(2-p)\ve_0(p)\le \ve<\ve_0(p),$
\eq[cep]{
C(\ve,p)=\frac{e^{-\ve/\ve_0(p)}}{1-\ve/\ve_0(p)};
}
and for all $0\le\ve < \frac{2}e,$  
\eq[ce1]{
\frac{e^{-\frac e2\ve}}{1-\frac e2\ve}\le C(\ve,1)\le \frac{1}{1-\frac e2\,\ve}.
}
\end{theorem}

We can finally complete the picture for all $p\ge1.$ Remarkably, the formula for $\ve_0(p)$ for the case $p>2$ is the same as for $1\le p\le 2,$ though it takes much more work to show.
\begin{theorem}
\label{t1}
For $p>2$,
\eq[maineps]{
\ve_0(p)=\left[\frac pe\Big(\Gamma(p)-\int_0^1t^{p-1}e^t\,dt\Big)+1\right]^{1/p}.
}
\end{theorem}
\cre{In contrast with the case $1<p\le 2,$ for $p>2$} we do not know the exact $C(\ve,p)$ \cre{for any $\ve.$} While we could estimate this constant in a manner somewhat similar to~\eqref{ce1}, the estimates we currently have seem much too implicit to be useful, so we omit them.

Without going into details, we mention another, more important difference between the cases $p\le 2$ and $p>2.$ It was shown in~\cite{p<2} that 
the constant $\ve_0(p)$ is attained in the weak-type John--Nirenberg inequality for $1<p\le 2$ (the case $p=1$ was treated in~\cite{korenovsky} and~\cite{lerner}, while the case $p=2$ had been previously addressed in~\cite{vv}). However, the method used to show this fact for $p\le2$ fails for $p>2,$ and we do not actually know if the constant is attained (though we conjecture that it is). 

\cre{On the other hand, still} another interesting result from~\cite{p<2} does go through for $p>2.$ Specifically, we have the following theorem, which extends to $p>2$ the main result of Corollary~1.5 from~\cite{p<2}. It is a sharp lower estimate for the distance in $\BMO$ to $L^\infty$ in the spirit of Garnett and Jones~\cite{gj}.
\begin{theorem}
\label{t2.5}
If $p>2,$ $Q$ is an interval, and $\varphi\in\BMO(Q),$ then
\eq[dist1]{
\inf_{f\in L^\infty(Q)}\|\varphi-f\|_{\BMO^p(Q)}\ge \frac{\ve_0(p)}{\min\{\ve_{\varphi},\ve_{-\varphi}\}},
}
and this inequality is sharp.
\end{theorem}

As explained in~\cite{p<2}, the main idea for computing $\ve_0(p)$ for $p\ne 2$ is to consider the dual problem: instead of estimating for which values of $\|\varphi\|_{\BMO^p}$ the exponential oscillation $\ave{e^{\varphi-\ave{\varphi}}}$ \cre{might become unbounded}, one estimates from below $\BMO^p$ oscillations of logarithms of $A_\infty$ weights and computes their asymptotics as the $A_\infty$-characteristic goes to infinity. This idea is formalized in the following general theorem.

Fix $p>0.$ \cre{For $C\ge1,$ let}
\eq[domain]{
\Omega_C=\{x\in\mathbb{R}^2: e^{x_1}\le x_2\le C\,e^{x_1}\}.
}
For an interval $Q\cre{,}$ \cre{$C\ge1,$} and every $x=(x_1,x_2)\in\Omega_C,$ let 
\eq[adm]{
E_{x,C,Q}=\{\varphi\in L^1(Q):~\av{\varphi}Q=x_1,~\av{e^\varphi}Q=x_2,~[e^\varphi]_{A_\infty(Q)}\le C\}.
}
\cre{We will call elements of $E_{x,C,Q}$ test functions.}
Define the following lower Bellman function:
\eq[bh]{
\bel{b}_{p,C}(x)=\inf\{\av{|\varphi|^p}Q:~\varphi\in E_{x,C,Q}\}.
}

\begin{theorem}[\cite{p<2}]
\label{main}
\cre{Take $p>0.$} Assume that there exists a family of functions $\{b_C\}_{C\ge1}$ such that for each~$C\cre{,}$ $b_C$ is defined on $\Omega_C,$ $b_C\le \bel{b}_{p,C},$ and $b_C(0,\cdot)$ is continuous on the interval $[1,C].$ Then
\eq[mt1]{
\ve^p_0(p)\ge \limsup_{C\to\infty} b_C(0,C).
}
\end{theorem}
\cre{Thus, to estimate $\ve_0(p),$ we need a suitable family $\{b_C\}_{C}$ of minorants of $\bel{b}_{p,C}.$ Just as was done in~\cite{p<2}, we actually find the functions $\bel{b}_{p,C}$ themselves, for all $p>2$ and all sufficiently large $C.$ We proceed as follows: in Section~\ref{new_bellman}, for each suitable choice of $p$ and $C$ we construct the so-called Bellman candidate, denoted $b_{p,C}.$ This construction is more delicate and more technical than the one in~\cite{p<2}, and we briefly discuss the challenges involved. The proof that $b_{p,C}\le \bel{b}_{p,C}$ constitutes Section~\ref{induction}. In Section~\ref{optimizers}, we obtain the converse inequality by demonstrating explicit test functions that realize the infimum in~\eqref{bh}. It is then an easy matter to prove Theorems~\ref{t1} and~\ref{t2.5}, and it is taken up in Section~\ref{nonB}.}

\section{The construction of the Bellman candidate}
\label{new_bellman}
\noindent For $R> 0,$ let
$$
\Gamma_R=\{x\in\mathbb{R}^2:~x_2=Re^{x_1}\}.
$$
Then the domain $\Omega_C$ from~\eqref{domain} is the region in the plane lying between $\Gamma_1$ and $\Gamma_C.$

\subsection{Discussion and preliminaries} \cre{As mentioned earlier, the construction of the Bellman candidate} given here for the case $p>2$ is different from and quite more involved than those presented in~\cite{p<2} for the cases $p=1$ and $p\in(1,2].$ However, the main goal is the same as before and simple to state: we are building the largest locally convex function $b$ on $\Omega_C$ satisfying the boundary condition $b(x_1,e^{x_1})=|x_1|^p.$ 

Let us briefly explain the similarities and differences between the cases $p\in(1,2]$ and $p>2$ (the case $p=1$ is different from both). In both cases the graph of the candidate $b$ is a convex ruled surface, which means that through each point on the graph there passes \cre{a straight-line segment} contained in the graph. \cre{The domain $\Omega_C$ then} splits into a collection of subdomains with disjoint interiors, $\Omega_C=\cup_jR_j,$ such that $b$ is twice differentiable and satisfies the homogeneous \ma equation $b_{x_1x_1}b_{x_2x_2}=b^2_{x_1x_2}$ in the interior of each $R_j.$ In addition, each $R_j$ is foliated by straight-line segments connecting two points of the boundary $\Gamma_1\cup\Gamma_C,$ and each point $x\in {\rm int}(R_j)$ lies on only one such segment, unless $b$ is affine in the whole $R_j.$ We call such segments \ma characteristics of $b.$ \cre{Typically,} if one knows the characteristics everywhere in $\Omega_C,$ one knows the function $b.$

\cre{Thus, to construct the candidate one has to understand} how to split $\Omega_C$ into subdomains and how to foliate each of them, so that the resulting function $b$ is \cre{locally} convex. \cre{If this is done while ensuring certain compatibility conditions, then  $b$ will almost automatically be the {\it largest} locally convex function with the given boundary conditions, as desired.} This is, in general, a difficult task, and the situation is further complicated by the fact that the splitting is \cre{usually} different for different $C.$

Fortunately, there is now a fairly general theory for building \cre{such} foliations on \cre{special non-convex domains such as ours.} It was started in~\cite{sv1} in the context of $\BMO^2;$ much developed \cre{and systematized} in the papers~\cite{iosvz1} and~\cite{iosvz2}, still for \cre{the parabolic strip of} $\BMO^2;$ and is now being adapted to general domains, such as $\Omega_C,$ in~\cite{isvz}. We also mention the recent paper~\cite{sz}, which formalized the theoretical link between Bellman functions and smallest locally concave (or largest locally convex, as is the case here) functions on \cre{the corresponding} domains.

A key building block for many \ma foliations is the tangential foliation. Let us explain this notion in our setting.

For $C\ge1,$ let $\xi=\xi(C)$ be the unique \cre{non-negative} solution of the equation 
$$
e^{-\xi}=C(1-\xi):\quad 0\le\xi<1.
$$
Note that $\xi(1)=0$ and that $\xi$ is strictly increasing with $\lim_{C\to\infty}\xi(C)=1.$
Let
\eq[km0]{
k(z)=\frac{e^z}{1-\xi},\quad z\in\mathbb{R},
}
and define a new function $u=u(x)$ on $\Omega_C$ by the implicit formula
\eq[udef]{
x_2=k(u)(x_1-u)+e^u.
}
This function has simple geometrical meaning, illustrated in Figure~\ref{fig_u_p>2}: if one draws the one-sided tangent to $\Gamma_C$ that passes through $x,$ so that the point of tangency is to the right of $x,$ then this tangent intersects $\Gamma_1$ at the point $(u,e^u),$ while the point of tangency is $(u+\xi,Ce^{u+\xi}).$ In particular $u(0,C)=-\xi.$ (We note that in~\cite{p<2}, $\xi$ and $u$ were called $\xi^+$ and $u^+,$ respectively). 

\begin{figure}[!h!]
\includegraphics[width=13cm]{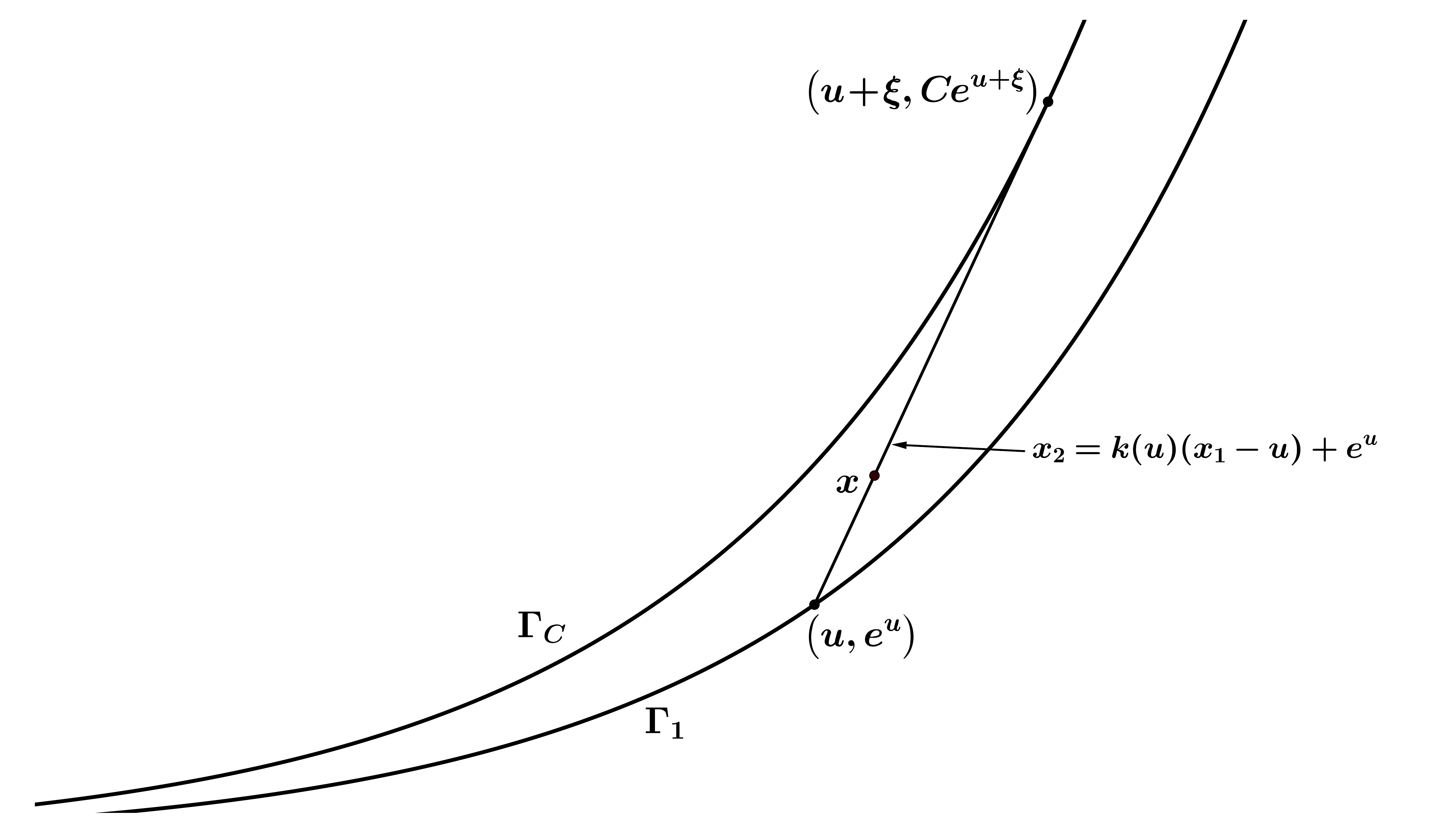}
\caption{The geometric meaning of $u(x)$ and $\xi$}
\label{fig_u_p>2}
\end{figure}

In the case $1<p\le 2,$ if $C$ was large enough, all of $\Omega_C$ was foliated by the tangents~\eqref{udef}, for $u\in(-\infty,\infty);$ \cre{thus, we did not have to split it into subdomains.} However, for $p>2,$ this uniform tangential foliation fails to yield a locally convex function on the whole $\Omega_C,$ for any $C.$ What actually happens --- and, again, only for sufficiently large $C$ ---  \cre{is shown on Figure~\ref{fig_p>2} later in this section. There we have two tangentially foliated subdomains, $R_1$ and $R_3,$ which are linked by a special ``transition regime'' consisting of two more subdomains: $R_4,$ where the candidate is affine and the foliation is thus degenerate, and $R_2,$ where the characteristics are chords connecting two points of $\Gamma_1.$ (In recent Bellman-function literature, these two particular shapes are called ``trolleybus''  and ``cup'', respectively; see~\cite{iosvz1,iosvz2,isvz}.)} This transition regime shrinks as $C$ grows, but never disappears. To show how all this fits together, we need some technical preparation. 

\subsection{Technical lemmas}
\begin{lemma}
\label{in_lunka}
~
\ben
\item[{\rm (1)}]
If $w>0$ and $v\in\big(-w,-w\,\frac{p-1}p\big),$ then 
\eq[ineq_lunka1]{
\frac{w^{p-1}+(-v)^{p-1}}{e^w-e^v}<(p-1)(-v)^{p-2}e^{-v}.
}
\item[{\rm (2)}]
If $0<w\le\frac{p-2}{p-1}$ and $v\in(-w,0),$ then 
\eq[ineq_lunka2]{
\frac{w^{p-1}+(-v)^{p-1}}{e^w-e^v}<(p-1)w^{p-2}e^{-w}.
}
\een
\end{lemma}
\begin{proof}
For Part (1), note that $e^{w-v}-1\ge w-v>0,$ and so
it is sufficient to check that
$$
w^{p-1}+(-v)^{p-1}<(p-1)(-v)^{p-2}(w-v)\,.
$$
Put $\theta=\cre{-\frac v{w}};$ then this inequality becomes
$$
(p-2)\theta^{p-1}+(p-1)\theta^{p-2}-1>0\,,\qquad\frac{p-1}p <\theta<1\,.
$$
The left-hand side is increasing in $\theta,$ and equals 
$
2\frac{(p-1)^{\cre{p}}}{p^{p-1}}-1
$
when $\theta=\frac{p-1}p.$ 
In turn, this is an increasing function of $p,$ equal to 0 at $p=2$. 

For Part (2), observe that since $w>-v,$ and $1-w\ge 1-\frac{p-2}{p-1}=\frac1{p-1},$ we have
$$
1-e^{-(w-v)}> (w-v)\Big[1-\frac12(w-v)\Big]>(w-v)(1-w)\ge\frac{w-v}{p-1},
$$
and \eqref{ineq_lunka2} follows from the obvious relation
$
w^{p-1}+(-v)^{p-1}<w^{p-2}(w-v).
$
\end{proof}

For any $v<0$ and $w>0,$ let
\eq[rq]{
r(u,w)=\frac{e^w-e^v}{w-v},\qquad q(v,w)=\frac{w^p-(-v)^p}{w-v}.
}
\begin{lemma}
\label{lemma_lunka}
For each $w\in\big(0,\frac{p-2}{3p}\big),$ there exists a unique $v\in\big(-w,-w\,\frac{p-1}p\big)$ such that
\eq[eq_lunka]{
\frac{q(v,w)+p(-v)^{p-1}}{r(v,w)-e^v}=\frac{pw^{p-1}-q(v,w)}{e^w-r(v,w)}=p\,\frac{w^{p-1}+(-v)^{p-1}}{e^w-e^v}.
}
\end{lemma}
\begin{proof}
Observe that it is enough to show only the first equality in~\eqref{eq_lunka}, as the second one then follows by elementary rearrangement. 
In turn, this first equality is equivalent to the statement
\eq[F]{
F(v,w):=(e^w-e^v)(w^p-(-v)^p-pw^{p-1}-p(-v)^{p-1})+
p(w-v)(w^{p-1}e^v+(-v)^{p-1}e^w)=0.
}
Assume $w\in(0,\frac{p-2}{3p})$ and let $\lambda=\frac{p-1}p.$ To show that there exists $v\in(-w,-\lambda w)$ such that $F(v,w)=0,$ we compare the signs of $F(-w,w)$ and $F(-\lambda w,w).$ 

Since $F(-w,w)=4pw^{p-1}(w\cosh w-\sinh w)>0,$ we want to check that $F(-\lambda w,w)<0.$ To that end,
\begin{align*}
F(-\lambda w,w)&=(e^w-e^{-\lambda w})\big[(1-\lambda^p)w^p-p(1+\lambda^{p-1})w^{p-1}\big]
+p(1+\lambda)w^p\big[e^{-\lambda w}+\lambda^{p-1}e^w\big],
\end{align*}
and the inequality $F(-\lambda w,w)<0$ is equivalent to
$$
\frac{(1+\lambda)w}{e^{(1+\lambda)w}-1}-1+\frac{1+p\lambda^{p-1}+(p-1)\lambda^p}
{p(1+\lambda^{p-1})}w<0.
$$
Let
$$
\psi(t)=\frac t{e^t-1}-1+\frac{1+p\lambda^{p-1}+(p-1)\lambda^p}
{p(1+\lambda)(1+\lambda^{p-1})}t\,.
$$
We would like to show that $\psi(t)<0$ for $t\in(0,\frac{p-2}{3p}(1+\lambda))$.
Note that for $t>0,$
$$
e^t>1+t+\frac12t^2\quad\implies\quad\frac t{e^t-1}-1<-\frac t{t+2}\,.
$$
Therefore, \cre{for $t>0,$}
$$
\psi(t)<-t\Big[\frac1{t+2}-\frac{1+p\lambda^{p-1}+(p-1)\lambda^p}{p(1+\lambda)
(1+\lambda^{p-1})}\Big],
$$
and it is sufficient to check that
$$
\frac{p-2}{3p}(1+\lambda)<\frac{p(1+\lambda)(1+\lambda^{p-1})}{1+p\lambda^{p-1}
+(p-1)\lambda^p}-2=
\frac{(p-2)(1-\lambda^p)+p\lambda(1-\lambda^{p-2})}{1+p\lambda^{p-1}+(p-1)\lambda^p}.
$$
Since
$
(p-2)(1-\lambda^p)+p\lambda(1-\lambda^{p-2})>
(p-2)(1+\lambda)(1-\lambda^{p-1})
$
and $\lambda^p<\lambda^{p-1},$
it is enough to verify that
$$
3>\frac{1/p+(2-1/p)\lambda^{p-1}}{1-\lambda^{p-1}}=\frac{2}{1-\lambda^{p-1}}-2+1/p.
$$
This is so because the right-hand side is decreasing in $p$ and equals $\frac52$ when $p=2.$ This proves that the desired $v$ exists for each $w.$

To show that $v$ is unique, we differentiate the function $F$ with respect to $v.$ This derivative can be written as follows: 
\begin{align*}
F_v(v,w)&=\big(e^w(w-v)-e^w+e^v\big)e^v\left(\frac{pw^{p-1}(w-v)-w^p+(-v)^p}{e^w(w-v)-e^w+e^v}-p(p-1)(-v)^{p-2}e^{-v}\right)\\
&=\big(e^w(w-v)-e^w+e^v\big)e^v\left(p\,\frac{w^{p-1}+(-v)^{p-1}}{e^w-e^v}-p(p-1)(-v)^{p-2}e^{-v}\right),
\end{align*}
where we used the second equality in~\eqref{eq_lunka}. Now, the first factor is positive, because the function $t\mapsto e^t$ is \cre{strictly} convex, while the last factor is negative by~\eqref{ineq_lunka1}.
\cre{Therefore, $F_v(v,w)$ is negative for any root $v$ of the equation $F(v,w)=0$ that lies in the interval $\big(-w,-w\,\frac{p-1}p\big),$ which is possible only when such a root is unique.}
\end{proof}

From now on, in using $v$ and $w$ we will always presume that $w\in(0,\frac{p-2}{3p}),$ $v\in(-w,-w\,\frac{p-1}p),$ and the pair $\{v,w\}$ is related by equation~\eqref{eq_lunka}. For such $v$ and $w,$ each of the three equal quantities in~\eqref{eq_lunka} \cre{is a function of $w,$} and it is convenient to give them a common name. Let
\eq[D]{
D(w)=\frac{q(v,w)+p(-v)^{p-1}}{r(v,w)-e^v}=\frac{pw^{p-1}-q(v,w)}{e^w-r(v,w)}=p\,\frac{w^{p-1}+(-v)^{p-1}}{e^w-e^v},
}
and $D$ is a function of $w$ defined on the interval $\big(0,\frac{p-2}{3p}\big).$ Let us list some of its properties.
\begin{lemma}
\label{lemma_d}
We have
\eq[ineq_d1]{
D(w)< p(p-1)(-v)^{p-2}e^{-v}
}
and
\eq[ineq_d2]{
D(w)< p(p-1)w^{p-2}e^{-w}.
}
Furthermore, \cre{$D'>0$ on} $\big(0,\frac{p-2}{3p}\big).$
\end{lemma}
\begin{proof}
Inequalities~\eqref{ineq_d1} and~\eqref{ineq_d2} come directly from~\eqref{ineq_lunka1} and~\eqref{ineq_lunka2}, respectively (note that $\frac{p-2}{3p}<\frac{p-2}{p-1},$ so~\eqref{ineq_lunka2} applies). 

\cre{To check the sign of $D,$} we will treat $q$ and $r$ as functions of $w$ and use the prime to indicate the total derivative with respect to $w.$ Thus, $q'=q_w+q_vv_w$ and $r'=r_w+r_vv_w,$ where $v_w$ can be computed from \cre{equation~ \eqref{F}.} Also denote $f(w):=w^p$ and $g(w):=e^w.$ 

We will need the simple 
but key fact that equation~\eqref{eq_lunka} can be written as
\eq[form]{
\frac{q'}{r'}=\frac{q-f'}{r-g'}=D.
}
Using this identity, we have
\begin{align*}
D'&=\Big(\frac{q-f'}{r-g'}\Big)'=\frac{(q'-f'')(r-g')-(q-f')(r'-g'')}{(r-g')^2}\\
&=\frac{g''(q-f')-f''(r-g')}{(r-g')^2}
=\frac{g''}{r-g'}\,\Big( \frac{q-f'}{r-g'}-\frac{f''}{g''}\Big).
\end{align*}
Since $g$ is \cre{strictly} convex, we have $g''>0$ and $r-g'<0.$ 
On the other hand, the expression in parentheses is negative by ~\eqref{ineq_d2}. 
\end{proof}
For $p>2,$ let
\eq[xicr]{
\xi_0(p)=\xicr,\qquad C_0(p)=\frac{e^{-\xi_0(p)}}{1-\xi_0(p)}.
}
\begin{lemma}
\label{lemma_tr}
Assume $\xi>\xi_0(p).$ Let 
\cre{$$
c_1=\xi\big[e(1-\xi)\Gamma(p-1)\big]^{1/(p-2)}, \qquad c_2=\xi\big[2e(1-\xi)\Gamma(p)\big]^{1/(p-2)}.
$$ }
Then the equation
\eq[tr0]{
\Big(\frac1\xi-1\Big)\int_{w}^\infty  s^{p-2}e^{-s/\xi}\,ds-w^{p-2}e^{-w/\xi}=0
}
has a unique solution $w_*$ in the interval $(0,c_1).$ 

Furthermore, the equation
\eq[trolleybus]{
\Big(\frac1\xi-1\Big)p(p-1)e^{\cre{w}(1/\xi-1)}\int_{w}^\infty s^{p-2}e^{-s/\xi}\,ds=D(w)
}
has a unique solution $\bcr$ in the interval $(w_*,c_2).$
\end{lemma}
\begin{proof}
First observe that $c_1<c_2<\xi.$ The first inequality is trivial, while the second is equivalent to
$
\xi>1-\frac1{2e\Gamma(p)},
$
which is clearly satisfied by the assumption $\xi>\xi_0(p).$ Second, we have $c_2<\frac{p-2}{3p}.$ Indeed, this inequality is equivalent to
$$
\xi>1-\frac{\big(1-\frac2p\big)^{p-2}}{2e\,(3\xi)^{p-2}\Gamma(p)}.
$$
Since $\xi<1$ and $\big(1-\frac2p\big)^{p-2}>e^{-2},$ this inequality is weaker than $\xi>1-\frac9{2e^3 3^p\Gamma(p)},$ which is in turn weaker than $\xi>\xi_0(p).$

Consider equation~\eqref{tr0}. When $w=0,$ the left-hand side of~\eqref{tr0} is positive. For $w=c_1,$
\begin{align*}
\Big(\frac1\xi-1\Big)&\int_{c_1}^\infty  s^{p-2}e^{-s/\xi}\,ds-c_1^{p-2}e^{-c_1/\xi}=
(1-\xi)\xi^{p-2}\int_{c_1/\xi}^\infty s^{p-2}e^{-s}\,ds-c_1^{p-2}e^{-c_1/\xi}\\
&
\cre{<(1-\xi)\xi^{p-2}\Gamma(p-1)-c_1^{p-2}e^{-c_1/\xi}
=\xi^{p-2}(1-\xi)\Gamma(p-1)\big(1-e^{1-c_1/\xi}\big)<0,}
\end{align*}
since $c_1<\xi.$ 
Thus, a solution $w_*\in(0,c_1)$ exists. To prove that it is unique, we note that the left-hand side of~\eqref{tr0} is decreasing in $w$ for $w\in(0,p-2),$ and that $c_1<p-2.$

Turning to~\eqref{trolleybus}, for $w=w_*$ we have, by~\eqref{tr0} and~\eqref{ineq_d2},
$$
\Big(\frac1\xi-1\Big)p(p-1)e^{w_*(1/\xi-1)}\int_{w_*}^\infty s^{p-2}e^{-s/\xi}\,ds=
p(p-1)w_*^{p-2}e^{-w_*}>D(w_*).
$$
Observe that for any $w,$ since $1-e^{v-w}<1-e^{-2w}<2w,$
$$
D(w)=p\,\frac{w^{p-1}+(-v)^{p-1}}{e^w-e^v}>pe^{-w}\frac{w^{p-1}}{1-e^{v-w}}>\frac12\,pe^{-w}w^{p-2}.
$$
Therefore, putting $w=c_2$ in \cre{the left-hand side of}~\eqref{trolleybus} we get
\begin{align*}
\Big(\frac1\xi-1\Big)p(p-1)e^{c_2(1/\xi-1)}\int_{c_2}^\infty s^{p-2}e^{-s/\xi}\,ds<
(1-\xi)\xi^{p-2}p\Gamma(p)e^{1-c_2}=\frac12\,pe^{-c_2}c_2^{p-2}<D(c_2),
\end{align*}
and\cre{, hence,} a solution $\bcr\in(w_*,c_2)$ exists. To prove uniqueness, observe that \cre{the derivative of the left-hand side of~\eqref{trolleybus} is a positive multiple of the left-hand side of~\eqref{tr0}; thus, it equals zero at $w=w_*$ and is decreasing for $w\in(0,p-2);$ in particular, it is negative for $w> w_*.$ 
Therefore, the left-hand side of~\eqref{trolleybus} is decreasing in $w$ for $w\ge w_*,$ while by Lemma~\ref{lemma_d}, the right-hand side is increasing.}
\end{proof}
\begin{remark}
In what follows, in addition to $\bcr,$ we will also use $\acr,$ which is the unique solution of the equation $F(v,\bcr)=0$ guaranteed by Lemma~\ref{lemma_lunka}.
\end{remark}
\subsection{The Bellman candidate.}
As mentioned earlier, we now split domain $\Omega_C$ into four subdomains, $\Omega_C=\cup_{j=1}^4 R_j.$ In addition to the numbers $\acr$ and $\bcr$ given by Lemma~\ref{trolleybus}, the definition below uses the function $k$ from~\eqref{km0} \cre{and function $r$ from~\eqref{rq}.} The splitting is pictured in Figure~\ref{fig_p>2}.
\begin{figure}[!h!]
\includegraphics[width=18cm]{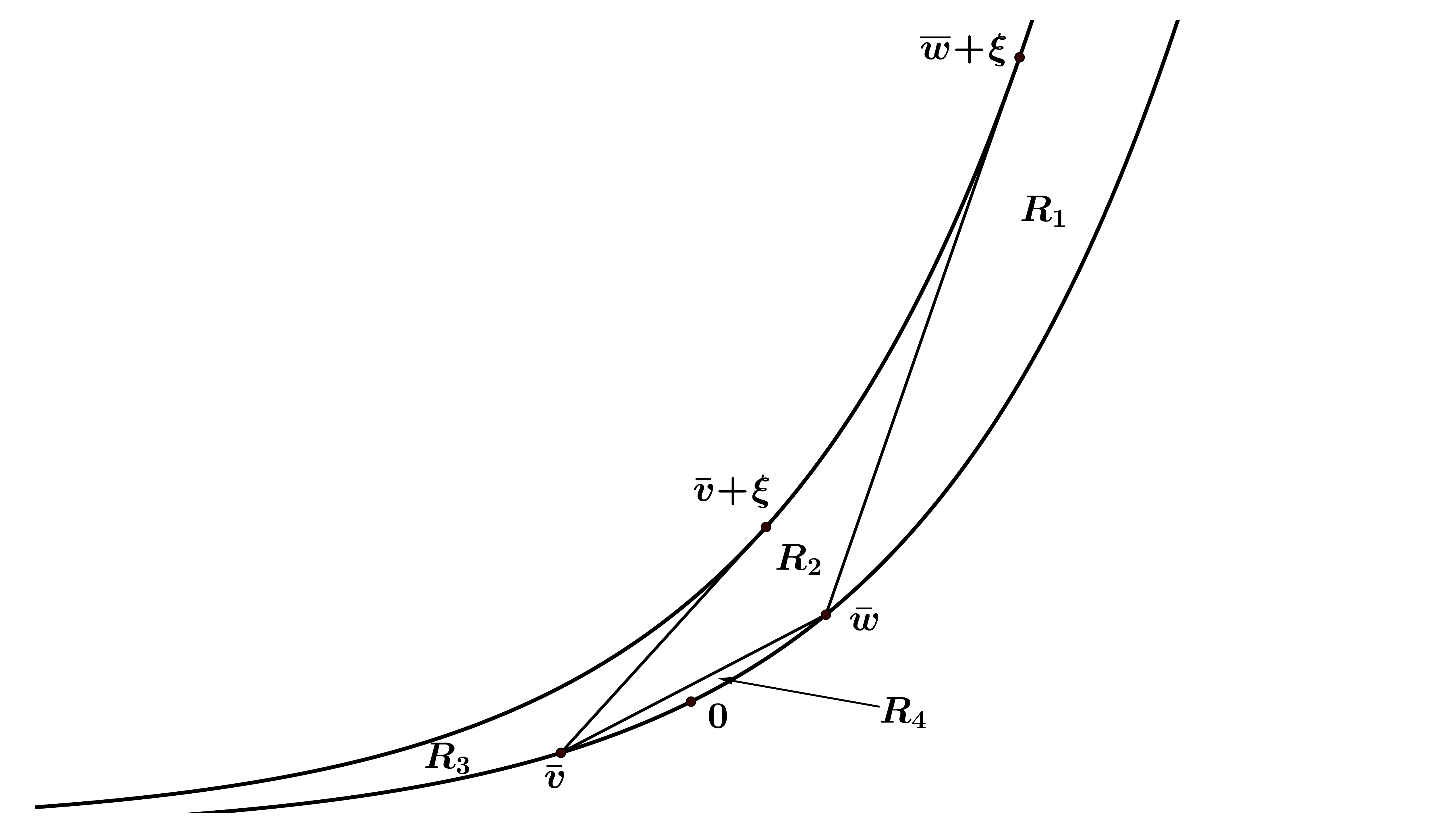}
\caption{The splitting of $\Omega_C$ for sufficiently large $C$}
\label{fig_p>2}
\end{figure}

\cre{\eq[bp>2domain]{
\begin{aligned}
R_1&=\big\{x\in\Omega_C\colon x_2 \le k(\bcr)(x_1-\bcr)+e^{\bcr}\big\} 
\cup\big\{x\in\Omega_C\colon x_1\ge\bcr+\xi\big\};
\\
R_2&=\big\{x\in\Omega_C\colon x_2\le k(\acr)(x_1-\acr)+e^{\acr},\ 
x_2\ge r(\acr,\bcr)(x_1-\bcr)+e^{\bcr},\ x_2\ge k(\bcr)(x_1-\bcr)+e^{\bcr}\big\}
\\
&\cup\big\{x\in\Omega_C\colon \acr+\xi\le x_1\le\bcr+\xi,\ x_2\ge k(\acr)(x_1-\acr)+e^{\acr},\ 
x_2\ge k(\bcr)(x_1-\bcr)+e^{\bcr}\big\};
\\
R_3&=\big\{x\in\Omega_C\colon x_1\le \acr+\xi,\ x_2\ge k(\acr)(x_1-\acr)+e^{\acr}\big\};
\\
R_4&=\big\{x\in\Omega_C\colon x_2\le r(\acr,\bcr)(x_1-\bcr)+e^{\bcr}\}.
\end{aligned}
}}

Our Bellman candidate will have a different expression in each of the four subdomains, requiring several auxiliary objects. First, for $z\in\mathbb{R},$ let
\eq[km1]{
m_1(z)=\frac p\xi\,e^{\cre{z}/\xi}\int_z^\infty s|s|^{p-2} e^{-s/\xi}\,ds,
}
and for $z<\acr,$ let
\eq[m3]{
m_3(z)=-\frac p\xi\,e^{z/\xi}\int_z^{\acr} (-s)^{p-1} e^{-s/\xi}\,ds+e^{(z-\acr)/\xi}\left(e^{\acr-\bcr}\big(m_1(\bcr)-p\bcr^{p-1}\big)-p(-\acr)^{p-1}\right).
}
The following intuitive lemma, whose simple proof is left to the reader, defines two new functions on $R_4.$
\begin{lemma}
\label{lemma_lunka_x}
For each $x=(x_1,x_2)\in R_4$ there exists a unique pair $\{v,w\}$ satisfying~\eqref{F} and such that the line segment connecting the points $(w,e^w)$ and $(v,e^v)$ passes through~$x.$ Thus,
$$
x_2=r(v,w)(x_1-w)+e^w.
$$
\cre{In the special case $x=(0,1)$ this segment degenerates into a point: $v=w=0.$}
\end{lemma}
From here on, we reserve the symbols $w$ and $v$ for the two functions on $R_4$ given by this lemma: $w=w(x)$ and $v=v(x);$ see Figure~\ref{fig_lunka}.
\begin{figure}[!h!]
\includegraphics[width=12cm]{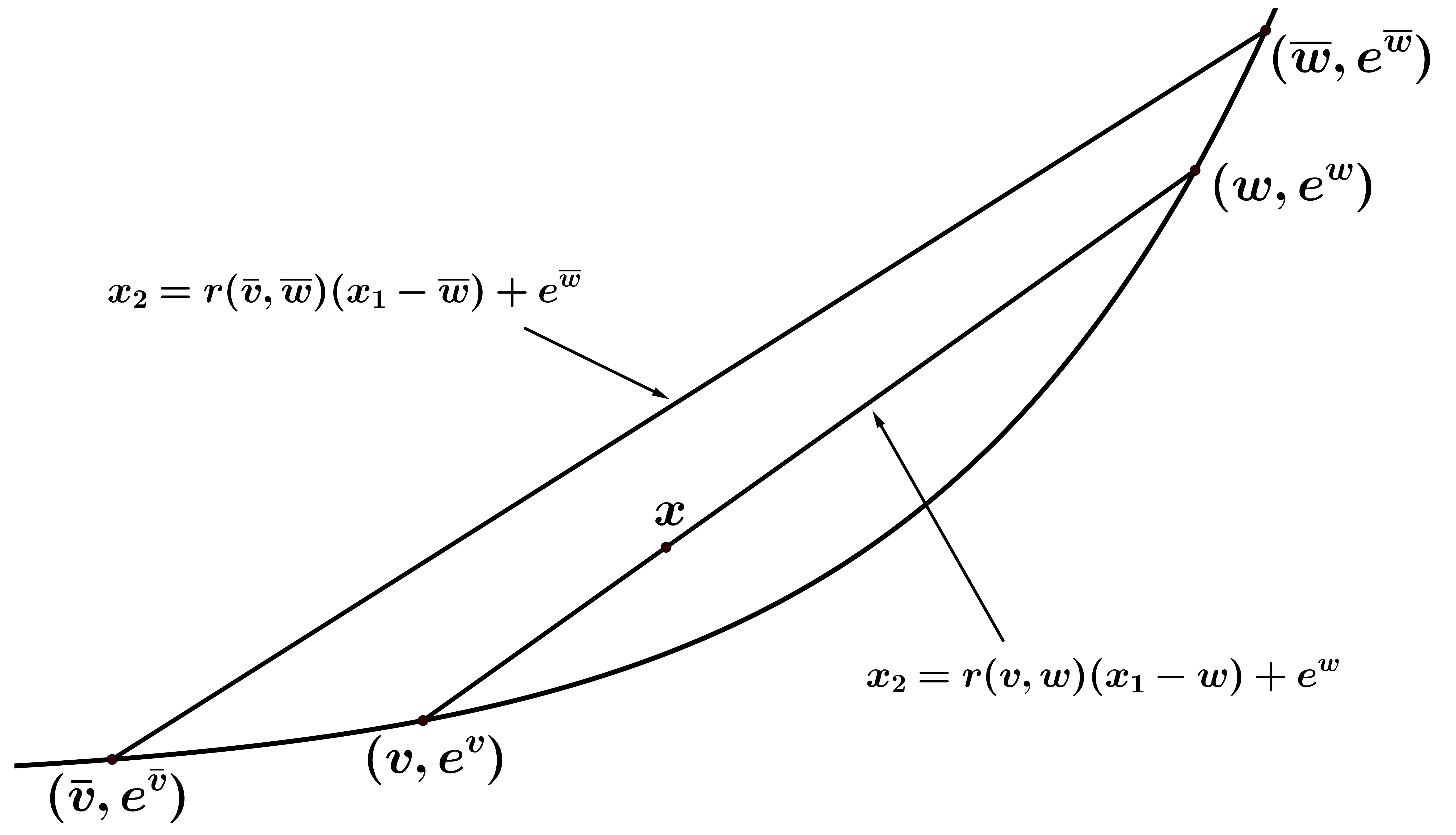}
\caption{The subdomain $R_4$ and the geometric meaning of $v(x)$ and $w(x)$}
\label{fig_lunka}
\end{figure}

Finally, here is our complete Bellman candidate. For $p>2$ and $C>C_0(p),$ let
\eq[b_main]{
b_{p,C}(x)=
\begin{cases}
m_1(u)(x_1-u)+u^p,& x\in R_1,
\bigskip

\\
q(\acr,\bcr)\,(x_1-\bcr)+\bcr^p+\frac{m_1(\bcr)-q(\acr,\bcr)}{k(\bcr)-r(\acr,\bcr)}\big(x_2-r(\acr,\bcr)(x_1-\bcr)-e^{\bcr}\big),&x\in R_2,
\bigskip

\\
m_3(u)(x_1-u)+(-u)^p,& x\in R_3,
\bigskip

\\
q\big(v,w)\,(x_1-w)+w^p,& x\in R_4.
\end{cases}
}
\cre{Recall that here $u=u(x)$ is given by~\eqref{udef}; $v=v(x)$ and $w=w(x)$ have just been defined in Lemma~\ref{lemma_lunka_x}; $k$ is given by~\eqref{km0}; $r$ and $q$ are given by~\eqref{rq}; $m_1$ is given by~\eqref{km1}; and $m_3$ is given by~\eqref{m3}. In addition, $\bcr$ was defined in Lemma~\ref{lemma_tr} as the solution of equation~\eqref{trolleybus}, while $\acr$ was defined in the remark following that lemma as the unique solution of equation~$F(v,\bcr)=0$ with $F$ given by equation~\eqref{F}.}

The next section presents the main theorem relating \cre{the candidate} $b_{p,C}$ and \cre{the Bellman function} $\bel{b}_{p,C}$ from~\eqref{bh}.

\section{The main Bellman theorem \cre{and a proof of the lower estimate}}
\label{induction}
The following result is the principal ingredient in the proofs of Theorems~\ref{t1} and~\ref{t2.5}.
\begin{theorem}
\label{main_bellman}
If $p>2$ and $C>C_0(p),$ then
\eq[mbb]{
\bel{b}_{p,C}= b_{p,C}\quad in\quad \Omega_C.
}
\end{theorem}

\cre{
As is common, we split the proof of Theorem~\ref{main_bellman} in two parts: the so-called direct inequality $\bel{b}_{p,C}\ge b_{p,C}$ and its converse.

\begin{lemma}
\label{direct}
If $p>2$ and $C>C_0(p),$ then
\eq[mbb1]{
\bel{b}_{p,C}\ge b_{p,C}\quad in\quad \Omega_C.
}
\end{lemma}
\begin{lemma}
\label{converse}
If $p>2$ and $C>C_0(p),$ then
\eq[mbb2]{
\bel{b}_{p,C}\le b_{p,C}\quad in\quad \Omega_C.
}
\end{lemma}
The proofs of Theorems~~\ref{t1} and~\ref{t2.5} use only Lemma~\ref{direct}, which we prove in this section. For the sake of completeness, we will show that the infimum in the definition of the Bellman function is attained at every point in $\Omega_C,$ and our candidate is in fact the Bellman function. This is done in Section~\ref{optimizers} where we prove Lemma~\ref{converse} .}

\cre{The analog of Lemma~\ref{direct}} for $p\in[1,2]$ was proved in Section~5 of~\cite{p<2}. In fact, the proof given there did not depend on the specific range of $p$ used. Rather, its main ingredient was showing that $b_{p,C}$ is locally convex in $\Omega_C,$ i.e., convex along every line segment contained in~$\Omega_C.$ More precisely, the main result of Lemma~5.1 from~\cite{p<2} can be restated as follows.
\begin{lemma}[\cite{p<2}]
\label{lemma_old}
Fix $p>0$ and assume that for some $C(p)\ge1$ there is a family of functions 
$\{b_{p,C}\}_{C\ge C(p)}$ satisfying the following conditions for each $C:$
\ben
\item[{\rm (1)}]
$b_{p,C}$ is locally convex in~$\Omega_C;$ 
\item[{\rm (2)}]
$b_{p,C}$ is continuous in~$\Omega_C;$
\item[{\rm (3)}]
For each $x\in\Omega_C,$ 
$$
\lim_{\cre{c\searrow C}}b_{p,c}(x)=b_{p,C}(x);
$$
\item[{\rm (4)}]
For each $s\in\mathbb{R},$ $b_{p,C}(s,e^s)=|s|^p.$
\een
Then for all $C\ge C(p),$
$$
\bel{b}_{p,C}\ge b_{p,C}\quad\text{in}\quad\Omega_C.
$$
\end{lemma}
\noindent It is \cre{routine} to check that conditions (2)-(4) are satisfied for $b_{p,C}$ from~\eqref{b_main}. Therefore, \cre{Lemma~\ref{direct}} will be proved once we have established the following result.
\begin{lemma}
\label{lemma_conv}
For $p>2$ and $C>C_0(p),$ the function $b_{p,C}$ is locally convex in~$\Omega_C.$
\end{lemma}
Let us fix $p>2$ and $C>C_0(p)$ and through the end of this section write simply $b$ for $b_{p,C}.$ 
Before proving Lemma~\ref{lemma_conv}, let us collect several useful facts from earlier work. First, as explained in~\cite{sv1} and~\cite{p<2}, showing that $b$ is locally convex in $\Omega_C$ is the same as showing that it is locally convex in each subdomain $R_k$ and that $b_{x_2}$ is increasing in $x_2$ across shared boundaries between subdomains. Second, in $R_1$ and $R_3,$ $b$ has the form
$$
b(x)=m(u)(x_1-u)+|u|^p, 
$$
where $m$ stands for $m_1$ or $m_3,$ respectively, and in each case satisfies the differential equation 
\eq[de]{
m'(u)=\frac1\xi(m(u)-pu|u|^{p-2}),
}
\cre{and $u=u(x)$ is given by~\eqref{udef}.} 
As shown in~\cite{p<2}, in such a case we have
\eq[bx2]{
b_{x_2}=m'(u)e^{-u}(1-\xi)
}
and also
\eq[ma]{
b_{x_1x_1}b_{x_2x_2}=b_{x_1x_2}^2,\qquad
\sign b_{x_2x_2}=\sign\big(m'(u)-m''(u)\big).
}
Therefore, to show that $b$ is locally convex in $R_1$ and $R_3$ we simply need to show that $m_1'(u)-m_1''(u)>0$ in $R_1$ and $m_3'(u)-m_3''(u)>0$ in $R_3.$

\begin{proof}[Proof of Lemma~\ref{lemma_conv}]
We first show local convexity of $b$ in each subdomain $R_k.$

In $R_1,$ a direct computation gives
$$
\frac{\xi^2}{p(p-1)}(m_1'(u)-m_1''(u))e^{-u/\xi}=\xi u^{p-2}e^{-u/\xi}-(1-\xi)\int_u^\infty e^{-s/\xi}s^{p-2}\,ds=:H_1(u),
$$
where $u\ge\bcr.$ We have 
$$
H'_1(u)=\xi u^{p-3}e^{-u/\xi}(p-2-u).
$$
Therefore, $H_1$ \cre{is increasing for $u\in(0,p-2)$ and decreasing for $u>p-2.$}  Since $H_1(u)\to0$ as $u\to\infty,$
to show that $H_1(u)>0$ for $u\ge\bcr,$ it suffices to show that $H_1(\bcr)>0.$ This immediately follows by applying first~\eqref{trolleybus} and then~\eqref{ineq_d2} with $w=\bcr:$
$$
(1-\xi)\int_{\bcr}^\infty e^{-s/\xi}s^{p-2}\,ds=\frac{\xi e^{\bcr(1-1/\xi)}}{p(p-1)}\,D(\bcr)< \xi \bcr^{p-2} e^{-\bcr/\xi}.
$$
Therefore, $b_{x_2x_2}>0$ in $R_1$ and so $b$ is locally convex in this subdomain.

In $R_2,$ $b$ is affine and thus locally convex.

In $R_3,$ we compute
\begin{align*}
&\frac{\xi^2}{p(p-1)}(m_3'(u)-m_3''(u))e^{-u/\xi}\\
&\qquad=\xi (-u)^{p-2}e^{-u/\xi}-(1-\xi)\Big(\int_u^{\acr} e^{-s/\xi}(-s)^{p-2}\,ds
+e^{(\bcr-\acr)(1/\xi-1)}\int_{\bcr}^\infty e^{-s/\xi}s^{p-2}\,ds\Big)\\
&\qquad=:H_3(u),
\end{align*}
where $u\le \acr.$ We have
$$
H'_3(u)=\xi (-u)^{p-3}e^{-u/\xi}(u-p+2)<0,
$$
and so to show that $H_3(u)\cre{>}0,$ it is enough to show that $H_3(\acr)\cre{>}0.$ Similarly to the case of $H_1,$ this follows from an application of \eqref{trolleybus} and then of~\eqref{ineq_d1} with $v=\acr:$
$$
(1-\xi)e^{(\bcr-\acr)(1/\xi-1)}\int_{\bcr}^\infty e^{-s/\xi}s^{p-2}\,ds=\frac{\xi e^{-\acr(1/\xi-1)}}{p(p-1)}\,D(\bcr)\cre{<} \xi (-\acr)^{p-2} e^{-\acr/\xi}.
$$
Thus, $b$ is locally convex in $R_3.$

Let us state the result for $R_4$ separately.
\begin{lemma}
\label{r4}
$b$ is convex in $R_4.$
\end{lemma}
\begin{proof}
In $R_4,$ $b$ is given by
$$
b(x)=q(v,w)(x_1-w)+f(w),\qquad x_2=r(v,w)(x_1-w)+g(w),
$$
where, as in the proof of Lemma~\ref{lemma_d}, we write $f(w)=w^p$ and $g(w)=e^w.$ Let us also, as we did there, use the prime to indicate the total derivative with respect to $w.$

To show that $b$ is convex, we show that $b_{x_1x_1}b_{x_2x_2}=b_{x_1x_2}^2$ and $b_{x_2x_2}>0$ in the interior of $R_4.$ \cre{Differentiating gives}
$$
w_{x_1}=\frac{-r}{r'(x_1-w)-r+g'}, \qquad w_{x_2}=\frac1{r'(x_1-w)-r+g'},
$$
and
\cre{
\begin{align*}
b_{x_1}&=\big[q'(x_1-w)-q+f'\big]w_{x_1}+q=-r\,\frac{q'(x_1-w)-q+f'}{r'(x_1-w)-r+g'}+q
=-rD+q,
\end{align*}
}
where \cre{we used}~\eqref{form}.
Similarly,
\eq[bx2short]{
\cre{b_{x_2}=\big[q'(x_1-w)-q+f'\big]w_{x_2}=\frac{q'(x_1-w)-q+f'}{r'(x_1-w)-r+g'}=D.}
}
Therefore,
$$
b_{x_1x_1}=-rD'w_{x_1},\quad b_{x_1x_2}=-rD'w_{x_2},\quad
b_{x_2x_2}=D'w_{x_2},
$$
and, since, $w_{x_1}=-rw_{x_2},$ we see that $b_{x_1x_1}b_{x_2x_2}=b_{x_1x_2}^2.$

Furthermore, since by Lemma~\ref{lemma_d}, \cre{$D'>0$} and since it is clear from geometry that $w_{x_2}>0,$ we have $b_{x_2x_2}>0,$ which completes the proof.
\end{proof}

To finish the proof of Lemma~\ref{lemma_conv}, we need to verify that $b_{x_2}$ is increasing \cre{in $x_2$} across boundaries between subdomains. We can write this requirement symbolically as:
$$
b_{x_2}\Big|^{}_{R_1, u=\bcr}\le b_{x_2}\Big|_{R_2},\qquad b_{x_2}\Big|^{}_{R_4, w=\bcr}\le b_{x_2}\Big|_{R_2},
\qquad b_{x_2}\Big|_{R_2}\le b_{x_2}\Big|^{}_{R_3, u=\acr}.
$$
In fact, all three statements hold with equality (which implies that $b$ is of class  $C^1$ in the interior of $\Omega_C,$ though we will not use this fact). 

From~\eqref{bx2short}, we have $b_{x_2}\cre{\big|^{}_{R_4, w=\bcr}}=D(\bcr).$ 
Using, in order,~\eqref{bx2},~\eqref{de},~\eqref{km1}, integration by parts, 
and~\eqref{trolleybus} gives:
\eq[use]{
b_{x_2}\Big|^{}_{R_1, u=\bcr}=m_1'(\bcr)e^{-\bcr}(1-\xi)=\frac1\xi(1-\xi)p(p-1)e^{\bcr(1/\xi-1)}\int_{\bcr}^\infty s^{p-2}e^{-s/\xi}\,ds=D(\bcr).
}
A very similar calculation, but using~\eqref{m3} in place of~\eqref{km1}, gives 
$b_{x_2}\big|^{}_{R_3, u=\acr}=D(\bcr).$

Finally, 
$$
b_{x_2}\Big|_{R_2}=\frac{m_1(\bcr)-q(\acr,\bcr)}{k(\bcr)-r(\acr,\bcr)}.
$$
By \eqref{de} and~\eqref{use},
$$
m_1(\bcr)=\xi m_1'(\bcr)+p\bcr^{p-1}=\frac{\xi}{1-\xi}\,e^{\bcr}D(\bcr)+p\bcr^{p-1}.
$$
Therefore,
$$
b_{x_2}\Big|_{R_2}=\frac{\frac{\xi}{1-\xi}\,e^{\bcr}D(\bcr)+p\bcr^{p-1}-q(\acr,\bcr)}{\frac1{1-\xi}\,e^{\bcr}-r(\acr,\bcr)}
=\frac{\frac{\xi}{1-\xi}\,e^{\bcr}D(\bcr)+\big(e^{\bcr}-r(\acr,\bcr)\big)D(\bcr)}{\frac1{1-\xi}\,e^{\bcr}-r(\acr,\bcr)}=D(\bcr).
$$
The proof is complete.
\end{proof}
We are now in a position to prove the main theorems stated in Section~\ref{main_results}.

\section{Proofs of Theorems~\ref{t1} and~\ref{t2.5}}
\label{nonB}
\noindent We will need two auxiliary results from~\cite{p<2}.

For $p>0,$ let
$$
\omega(p)=\left[\frac pe\Big(\Gamma(p)-\int_0^1t^{p-1}e^t\,dt\Big)+1\right]^{1/p}.
$$
\begin{lemma}[\cite{p<2}]
\label{phi0}
Let $\varphi_0(t)=\log(1/t),~t\in(0,1).$ Then
\eq[eps_phi]{
\ve_{\varphi_0}=1,\qquad \ve_{-\varphi_0}=\infty.
}
If $p\ge1,$ then
\eq[norm]{
\|\varphi_0\|_{\BMO^p((0,1))}=\omega(p).
}
Consequently,
\eq[norm1]{
\ve_0(p)\le \omega(p).
}
\end{lemma}
\begin{lemma}[\cite{p<2}]
 \label{cont}
 Let $\varphi$ be a non-constant $\BMO$ function. For $\ve\in[0,\ve_\varphi),$ let $F(\ve)=[e^{\ve\varphi}]_{A_\infty}.\!\!$ Then $F$ is a strictly increasing, continuous function on $[0,\ve_\varphi),$ and $\lim_{\ve\to\ve_\varphi}F(\ve)\!=\!\infty.$
 \end{lemma}

\begin{proof}[Proof of Theorem~\ref{t1}]
\cre{We use Theorem~\ref{main} with $b_C=b_{p,C}$ given by formula~\ref{b_main}. By Lemma~\ref{direct}, $b_C\le \bel{b}_{p,C},$ as required.} 

We need to compute $b_{p,C}(0,C).$ Note that by 
\cre{Lemma~\ref{lemma_tr}}, $\bcr<\xi,$ and, thus, $\acr>-\bcr>-\xi$ by Lemma~\ref{lemma_lunka}. Therefore, the point $(0,C)$ is in subdomain $R_3$ and, since $u(0,C)=-\xi,$
$$
b_{p,C}(0,C)=m_3(-\xi)\xi+\xi^p.
$$
Now, $m_3(-\xi)$ is given by~\eqref{m3}:
$$
m_3(-\xi)=-\frac p\xi\,e^{-1}\int_{-\xi}^{\acr} (-s)^{p-1} e^{-s/\xi}\,ds+e^{(-\xi-\acr)/\xi}\left(e^{\acr-\bcr}\big(m_1(\bcr)-p\bcr^{p-1}\big)-p(-\acr)^{p-1}\right).
$$
By \cre{Lemma~\ref{lemma_tr}}, $\bcr\in(0,c_2)$ with $c_2\to 0$ as $\xi\to1.$ By Lemma~\ref{lemma_lunka}, $\acr\in(-\bcr,0).$ Therefore, 
$$
\lim_{\xi\to1}\acr=\lim_{\xi\to1}\bcr=0
$$
and
$$
\lim_{C\to\infty}b_{p,C}(0,C)=\lim_{\xi\to1}\big(m_3(-\xi)\xi+\xi^p\big)=-\frac pe\int_{-1}^0(-s)^{p-1}e^{-s}\,ds+e^{-1}m_1(0)+1=\omega^p(p),
$$
where we used~\eqref{km1}, whereby \cre{$m_1(0)=p\Gamma(p).$}

Hence, by Theorem~\ref{main}, $\ve_0(p)\ge\omega(p),$ 
and Lemma~\ref{phi0} now finishes the proof.
\end{proof}

The proof of Theorem~\ref{t2.5} below is a variation of the argument for Corollary~1.5 in~\cite{p<2}; the proof of sharpness, using function $\varphi_0$ from Lemma~\ref{phi0}, is exactly the same and we omit it.
\begin{proof}[Proof of Theorem~\ref{t2.5}]
Take any $\varphi\in\BMO(Q).$ Without loss of generality, assume $\ve_\varphi<\infty.$ For $\ve\in[0,\ve_\varphi),$ let $F(\ve)=[e^{\ve\varphi}]_{A_\infty(Q)}.$ By Lemma~\ref{cont}, for sufficiently large $\ve$ we have $F(\ve)\ge C_0(p).$ Therefore, for any subinterval $J$ of $Q,$
$$
\av{|\ve\varphi-\av{\ve\varphi}J|^p}J\ge\bel{b}_{p,F(\ve)}\big(0,\av{e^{\ve\varphi-\sav{\ve\varphi}J}}J\big)\ge
b_{p,F(\ve)}\big(0,\av{e^{\ve\varphi-\sav{\ve\varphi}J}}J\big).
$$
Take a sequence $\{J_n\}$ such that $\av{e^{\ve\varphi-\sav{\ve\varphi}{J_n}}}{J_n}\to F(\ve).$ Since the left-hand side is bounded from above by $\ve^p\|\varphi\|^p_{\BMO^p(Q)},$ we have
$$
\ve^p\|\varphi\|^p_{\BMO^p(Q)}\ge b_{p,F(\ve)}\big(0,F(\ve)\big).
$$
Now, take $\ve\to\ve_\varphi$ (and, thus, $F(\ve)\to\infty).$ This gives
$$
\ve^p_\varphi \|\varphi\|^p_{\BMO^p(Q)}\ge \ve^p_0(p).
$$
Take any $f\in L^\infty(Q),$ then $\ve_{\varphi-f}=\cre{\ve_\varphi}.$ Thus, we can replace $\varphi$ with $\varphi-f$ above, which gives
$$
\|\varphi-f\|_{\BMO^p(Q)}\ge \frac{\ve_0(p)}{\ve_\varphi}.
$$
The same inequality holds with $\varphi$ replaced with $-\varphi,$ and it remains to take the infimum over $f\in L^\infty(Q)$ on the left.
\end{proof}

\section{Optimizers and the converse inequality}
\label{optimizers}
In this section, we complete the proof of Theorem~\ref{main_bellman} by proving Lemma~\ref{converse}. To that end, we present a set of special test functions on the interval $(0,1)$ that realize the infimum in definition~\eqref{bh} of the Bellman function $\bel{b}_{p,C}.$

Without loss of generality assume $C>1.$ Let $Q=(0,1).$ Recall the Bellman candidate $b_{p,C}$ given by formula~\eqref{b_main}. For $x\in\Omega_C,$ we say that a function $\varphi_x$ on $Q$ is an optimizer for $b_{p,C}$ at $x$ if
\eq[optimizer1]{
\varphi_x\in E_{x,C,Q}\quad\text{and}\quad \av{|\varphi_x|^p}{Q}=b_{p,C}(x),
}
where the set of test functions $E_{x,C,Q}$ is defined by~\eqref{adm}. Observe that if we have such a function $\varphi_x$ for all $x\in\Omega_C,$ then
$$
b_{p,C}(x)=\av{|\varphi|^p}Q\ge \bel{b}_{p,C}(x),
$$
which is the statement of Lemma~\ref{converse}.

Our optimizers $\varphi_x$ will have different forms depending on the location of $x$ in 
$\Omega_C.$ Specifically, we will have a different optimizer for each of the four subdomains $R_k$ of 
$\Omega_C$ defined by formula~\eqref{domain} and pictured in Figure~\ref{fig_p>2}. We do not discuss the construction of these optimizers, but simply give formulas for them. A reader interested in where they come from is invited to consult papers~\cite{sv1} and~\cite{iosvz2}, where a number of similar constructions are carried out in the context of $\BMO^2.$

For each $x\in R_1,$ let
\eq[opt_r1]{
\varphi_x(t)=u +\xi\log\big({\textstyle\frac{\alpha}t}\big)\chi^{}_{(0,\alpha)}(t),
}
where $u=u(x)$ is defined by~\eqref{udef} and we set
\eq[opt_r1a]{
\alpha=\frac{x_1-u}{\xi}.
}
(This optimizer was defined in Section~5 of~\cite{p<2} under the name $\varphi^+_x$.)

Now consider the subdomain $R_2.$ Let us give names to its four corners, clockwise from top right:
$$
X=(\bcr+\xi,e^{\bcr+\xi}),\qquad Y=(\bcr,e^{\bcr}),\qquad Z=(\acr,e^{\acr}),\qquad W=(\acr+\xi,e^{\acr+\xi}).
$$
We already know the optimizers for the points $X,$ $Y,$ and $Z:$ the first comes from formula~\eqref{opt_r1} (which applies since $X\in R_1\cap R_2$) with $\alpha=1;$ the other two are trivial, since for each $x\in\Gamma_1,$ the set $E_{x,C,Q}$ contains only one element --- the constant function $\varphi(t)=x_1.$ Therefore, we define, for all $t\in Q,$
$$
\varphi^{}_X(t)=\xi\log\big({\textstyle\frac1t}\big),\qquad \varphi^{}_Y(t)=\bcr,\qquad \varphi^{}_Z(t)=\acr.
$$
We now use these three optimizers to define $\varphi_x$ for every $x\in R_2.$ Observe that $R_2$ is contained in the triangle with the vertices $X,$ $Y,$ $Z.$ This means that every $x$ in $R_2$ has a unique representation as a convex combination of these three points. Thus, there are non-negative numbers $\alpha_1,$ $\alpha_2,$ and  $\alpha_3$ such that $\alpha_1+\alpha_2+\alpha_3=1$ and
\eq[cc]{
x=\alpha_1 X+\alpha_2 Y+\alpha_3 Z.
}
To obtain $\varphi_x,$ we concatenate $\varphi^{}_X,$ $\varphi^{}_Y,$ and $\varphi^{}_Z$ in the appropriate proportion:
$$
\varphi_x(t)=\varphi^{}_X\big({\textstyle\frac t{\alpha_1}}\big)\chi^{}_{(0,\alpha_1)}(t)+
\varphi^{}_Y\big({\textstyle\frac {t-\alpha_1}{\alpha_2}}\big)\chi^{}_{(\alpha_1,\alpha_1+\alpha_2)}(t)+
\varphi^{}_Z\big({\textstyle\frac {t-\alpha_1-\alpha_2}{\alpha_3}}\big)\chi^{}_{(\alpha_1+\alpha_2,1)}(t),
$$
or, equivalently,
\eq[opt_r2]{
\varphi_x(t)=\bcr\,\chi^{}_{(0,\alpha_1+\alpha_2)}(t)+\xi\log\big({\textstyle\frac{{\alpha_1}}t}\big)\chi^{}_{(0,\alpha_1)}(t)+\acr\,\chi^{}_{(\alpha_1+\alpha_2,1)}(t),
}
with $\alpha_k=\alpha_k(x)$ defined by~\eqref{cc}.

This formula applies, in particular, to the fourth corner of $R_2,$ i.e., the point $W.$ Since that point also lies in the subdomain $R_3,$ the optimizer $\varphi^{}_W$ will enter into the definition of $\varphi_x$ for all $x\in R_3.$ Specifically, for each such $x$ we set:
\eq[opt_r3]{
\varphi_x(t)=\varphi^{}_W\big({\textstyle\frac t{\tau\alpha}}\big)\chi^{}_{(0,\tau\alpha)}(t)+
\xi\log\big({\textstyle\frac{\alpha}t}\big)\chi^{}_{(\tau\alpha,\alpha)}(t)+u\chi^{}_{(\tau\alpha,1)}(t).
}
Here, $u$ is defined by~\eqref{udef}, $\alpha$ is defined by~\eqref{opt_r1a}, and we also set
\eq[opt_r3t]{
\tau=e^{(u-\acr)/\xi}.
}

It remains to define $\varphi_x$ when $x\in R_4.$ Recall the two auxiliary functions $v=v(x)$ and $w=w(x)$ defined by Lemma~\ref{lemma_lunka_x} (see Figure~\ref{fig_lunka}). Every point $x\in R_4\setminus\Gamma_1$ lies on the line segment connecting the points $(v,e^v)$ and $(w,e^w).$ Accordingly, we define $\varphi_x$ to be the appropriate concatenation of the two consant optimizers corresponding to those points:
\eq[opt_r4]{
\varphi_x(t)=w\,\chi^{}_{(0,\beta)}(t)+v\,\chi^{}_{(\beta,1)}(t),
}
where we set
\eq[opt_r4b]{
\beta=\frac{x_1-v}{w-v}.
}
We now state the main lemma, which immediately yields Lemma~\ref{converse}
\begin{lemma}
\label{main_opt}
Let $\varphi_x$ be defined by~\eqref{opt_r1} and~\eqref{opt_r1a} for $x\in R_1;$ by~\eqref{cc} and~\eqref{opt_r2} for $x\in R_2;$ by~\eqref{opt_r3},~\eqref{opt_r1a}, and~\eqref{opt_r3t} for $x\in R_3;$ and by~\eqref{opt_r4} and~\eqref{opt_r4b} for $x\in R_4.$ Then $\varphi_x$ is an optimizer for $b_{p,C}$ at every $x\in\Omega_C.$
\end{lemma}
\begin{remark}
If a point $x$ lies on a boundary shared by two subdomains, $\varphi_x$ seems to be defined by two different formulas. However, as is easy to check, in all cases above, such two formulas give exactly the same function.
\end{remark}
The proof of this lemma is very similar to the proof of Lemma~5.2 from~\cite{p<2} and we leave it to the reader.

\end{document}